\documentclass[12pt]{amsart}

\usepackage{tikz}
\usetikzlibrary{arrows,positioning}
\usepackage{amsmath}
\usepackage{amssymb}
\usepackage{amscd}
\usepackage{latexsym}
\usepackage{amsthm}
\usepackage{mathrsfs}
\usepackage{color}
\usepackage{setspace}
\newtheorem{thm}{Theorem}[section]
\newtheorem*{thm*}{Theorem}
\newtheorem{cor}[thm]{Corollary}
\newtheorem*{cor*}{Corollary}

\newtheorem*{lem*}{Lemma}
\newtheorem{prop}[thm]{Proposition}
\newtheorem*{prop*}{Proposition}

\theoremstyle{definition}
\newtheorem{defn}{Definition}[section]
\newtheorem*{defn*}{Definition}
\newtheorem*{conjecture}{Conjecture}

\theoremstyle{remark}
\newtheorem{rem}{Remark}[section]
\newtheorem*{rem*}{Remark}
\newtheorem{example}{Example}[section]

\newtheorem*{problem*}{Problem}

\newtheorem*{question*}{Question}

\newcommand{\Q}{\mathbb Q}
\newcommand{\bH}{\mathbb H}

\newcommand{\Z}{\mathbb Z}
\newcommand{\PP}{\mathbb P}
\newcommand{\cP}{\mathcal P}

\newcommand{\wt}{\widetilde}

\DeclareMathOperator{\pExp}{Exp}
\DeclareMathOperator{\pLog}{Log}

\DeclareMathOperator{\Sym}{Sym}

\usepackage[backend=bibtex,style=alphabetic,isbn=false,url=false,maxnames=5,firstinits=true]{biblatex}
\usepackage{amsfonts}
\usepackage{enumitem}
\renewbibmacro{in:}{}

\title{Plethystic identities and mixed Hodge structures of character varieties}
\author{Anton Mellit}
\date{\today}
\begin{document}

\begin{abstract}
I demonstrate how certain identities for Macdonald's polynomials established by Garsia, Haiman and Tesler, together with the conjecture of Hausel, Letellier and Villegas imply explicit relations between mixed Hodge polynomials of different character varieties. 
\end{abstract}

\maketitle

\onehalfspacing

\section{Symmetric functions}
The set of all partitions is denoted $\cP$ and we have the following symmetric functions for each $n\in\Z_{>0}$, $\lambda\in\cP$, say in the variables $x_1,x_2,\ldots$.
\begin{align*}
&\text{elementary:} & e_n =& \sum_{i_1<\cdots<i_n} x_{i_1}\ldots x_{i_n}, & e_\lambda&=\prod e_{\lambda_i}\\
&\text{homogeneous:} & h_n =& \sum_{i_1+i_2+\cdots=n} x_1^{i_1} x_2^{i_2} \cdots, & h_\lambda&=\prod h_{\lambda_i}\\
&\text{monomial:} & m_\lambda=& \sum_{v\in \text{permutations of $\lambda_1,\ldots,\lambda_k,0,0,\ldots$}} x_1^{v_1} x_2^{v_2} \cdots,\\
&\text{power sums:} & p_n =& \sum_{i} x_i^n, & p_\lambda&=\prod p_{\lambda_i}\\
&\text{Schur:} & s_\lambda=&\sum_{\text{$T$ a semistandard tableaux of shape $\lambda$}} x^T.
\end{align*}
We set $h_0=e_0=1$. $p_0$ is undefined.

The ring of symmetric functions is denoted by $\Sym$ and over $\Q$ it is just the polynomial ring in the variables $p_1, p_2,\ldots$.

\section{$\lambda$-rings}
All our rings are commutative and contain $\Q$.
\begin{defn}
A $\lambda$-ring (over $\Q$) is a ring $R$ endowed with the following operation. For each $x\in R$ and $F\in\Sym$ we have $F[x]\in\Sym$ in such a way that
\begin{enumerate}
\item For each $x\in R$ the map $F\to F[x]$ is a ring homomorphism from $\Sym$ to $R$.
\item For each $n\in\Z_{>0}$ the map $x \to p_n[x]$ from $R$ to $R$ is a ring homomorphism.
\item We have $p_m[p_n[x]]=p_{mn}[x]$ and $p_1[x]=x$ for each $m,n\in\Z_{>0}$ and $x\in R$.
\end{enumerate}
\end{defn}

\begin{rem}
Equivalently, we could define a $\lambda$-ring to be a ring together with a family of commuting ring homomorphisms $p_2, p_3,\ldots$ indexed by prime numbers.
\end{rem}

\begin{example}
We endow the polynomial ring $\Q[x_1,x_2,\ldots,x_k]$ with the \emph{standard} $\lambda$-structure by the rule $p_n[x_i]=x_i^n$ for each $n\in\Z_{>0}$ and $i=1,2,\ldots,k$. Similarly for the rings of rational functions, power series and Laurent power series, or rings of mixed type, i.e. rings like $\Q(q,t)[x_1,x_2,\ldots]$. It is common to consider only $\lambda$-rings with the standard $\lambda$-structure and call the operation $F\to F[x]$ \emph{plethystic substitution}.
\end{example}

\begin{example}
Any ring $R$ can be endowed with the \emph{stupid} $\lambda$-structure by setting $p_n[x]=x$ for all $n\in\Z_{>0}$.
\end{example}

\begin{defn}
An element $x$ in a $\lambda$-ring $R$ is called \emph{monomial} if $p_n[x]=x^n$ for each $n\in\Z_{>0}$.
\end{defn}

Monomial elements possess the following properties, which are easily verified for power sums and then extended to all symmetric functions.
\begin{prop}\label{prop:ususaleval}
For a sequence of monomial elements $x_1,x_2,\ldots$ in a $\lambda$-ring $R$ and any symmetric function $F$ we have
$$
F[x_1+x_2+\cdots] = F(x_1,x_2,\ldots).
$$
On the left hand side we apply the operation $F[-]$ to the element $x_1+x_2+\cdots$, while on the right hand side we view $F$ as a polynomial and evaluate it on $x_1,x_2,\ldots$.
\end{prop}
\begin{prop}
For a monomial element $u\in R$, any $F\in\Sym$ homogeneous of degree $d$, and $X\in R$ we have 
$$
F[uX] = u^d F[X].
$$
\end{prop}

\begin{defn}[Free $\lambda$-rings]
The free $\lambda$-ring in generators $X_1,X_2,\ldots,X_k$ will be denoted $\Sym[X_1,X_2,\ldots,X_k]$. It is simply the $k$-th tensor power of $\Sym$. It can be thought of as the ring of symmetric functions in $k$ sets of variables. Capital letters $X, Y, Z, W$ are reserved to denote arbitrary elements in arbitrary $\lambda$-rings, or equivalently, generators of free $\lambda$-rings. This is contrary to small letters, which will denote elements with special properties with respect to the $\lambda$-structure, usually monomial elements.
\end{defn}

\begin{prop}[Basic identities]\label{prop:basicident}
For each $n\in\Z_{\geq 0}$, we have
$$
h_n[-X]=(-1)^n e_n[X],
$$
$$
h_n[X+Y] = \sum_{i=0}^n h_i[X] h_{n-i}[Y],\qquad e_n[X+Y] = \sum_{i=0}^n e_i[X] e_{n-i}[Y],
$$
$$
h_n[XY] = \sum_{\lambda\vdash n} h_\lambda[X] m_\lambda[Y].
$$
\end{prop}

\subsection{Examples of computations}
We first evaluate $e_n[k]$ for each $k\in\Z_{>0}$. By Proposition \ref{prop:ususaleval} we have
$$
e_n[k] = e_n(1,1,\ldots,1)=\#\{1\leq i_1<\cdots< i_n\leq k\}=\binom{k}{n}. 
$$
Consider $\Q[x]$ with the stupid $\lambda$-ring structure $p_i[x]=x$. Then $e_n[x]$ must be a polynomial which agrees with $\binom{k}{n}$ for each $x=k\in\Z_{>0}$. Therefore
\begin{equation}\label{eq:binom1}
e_n[x] = \binom{x}{n} = \frac{x(x-1)\cdots (x-k+1)}{k!}\qquad (p_i[x]=x).
\end{equation}

Now let's take $R=\Q(q,u)$ with the standard $\lambda$-structure $p_i[q]=q^i$, $p_i[u]=u^i$ and try to evaluate $e_n\left[\frac{1-u}{1-q}\right]$. On one hand we have
$$
e_n\left[\frac{1-uq}{1-q}\right] = e_n\left[1+q\frac{1-u}{1-q}\right] = q^n e_n\left[\frac{1-u}{1-q}\right] + q^{n-1} e_{n-1}\left[\frac{1-u}{1-q}\right].
$$
On the other hand
$$
e_n\left[\frac{1-uq}{1-q}\right] = e_n\left[u+\frac{1-u}{1-q}\right] = e_n\left[\frac{1-u}{1-q}\right] + u e_{n-1}\left[\frac{1-u}{1-q}\right].
$$
Thus we have
$$
e_n\left[\frac{1-u}{1-q}\right] = \frac{q^{n-1}-u}{1-q^n}\; e_{n-1}\left[\frac{1-u}{1-q}\right] = \cdots=\frac{(1-u)\cdots(q^{n-1}-u)}{(1-q)\cdots(1-q^n)}.
$$
We can also compute
$$
h_n\left[\frac{1-u}{1-q}\right] = (-1)^n u^n e_n\left[\frac{1-u^{-1}}{1-q}\right] = \frac{(1-u)\cdots(1- u q^{n-1})}{(1-q)\cdots(1-q^n)}
$$
We can make a substitution $u=q^k$ because $q^k$ is monomial and obtain
\begin{equation}\label{eq:binom2}
e_{n}\left[\frac{q^k-1}{q-1}\right] = q^{\binom{n}{2}} \binom{k}{n}_q.
\end{equation}
For $q=1$ this also gives the classical binomial coefficient. We see that different deformations of the classical binomial coefficients are simply evaluations of $e_n$ in different $\lambda$-rings.

\begin{rem}
This example also illustrates the fact that one cannot do arbitrary substitutions in $\lambda$-ring identities. For instance, if a formula is proved under the assumption $p_i[q]=q^i$ on can only do substitutions $q=a$ such that $a$ also satisfies $p_i[a]=a^i$. Thus (\ref{eq:binom2}) is true if we replace $q$ by $1$, or $0$, or $q^{-1}$, but not something completely random. Similarly, we cannot set $x=q$ in (\ref{eq:binom1}). On the other hand, if we have an identity with free variables, like the ones in Proposition \ref{prop:basicident}, we can substitute arbitrary values.
\end{rem}

\section{Plethystic exponential}
\begin{defn}
$$
\pExp[X]:=\sum_{n=0}^\infty h_n[X]=\exp\left(\sum_{n=1}^\infty \frac{p_n[X]}{n}\right).
$$
\end{defn}

The series is understood in the completion of $\Sym[X]$ with respect to the degree. Several properties are immediate from Proposition \ref{prop:basicident}.
\begin{prop}
We have
$$
\pExp[X+Y] = \pExp[X]\pExp[Y],\qquad
\pExp[X Y] = \sum_{\lambda\in\cP} h_\lambda[X] m_\lambda[Y].
$$
\end{prop}

\begin{rem}
If $x$ is such that $p_n[x]=0$ for $n>1$, then the plethystic exponential becomes the usual exponential:
$$
\pExp[x]=\exp(x).
$$
If $k$ is such that $p_n[k]=k$ for all $n$, then for any $X$
$$
\pExp[kX] = (\pExp[X])^k.
$$
For monomial $x_1,x_2,\ldots,x_n$ we have
$$
\pExp\left[\sum_{i=1}^n x_i\right] = \prod_{i=1}^n \frac{1}{1-x_i}.
$$
\end{rem}

The operation $\pExp$ can be inverted. The inverse operation is called \emph{plethystic logarithm} and is denoted by $\pLog$. It is defined for series with constant term $1$.

\subsection{Hall pairing}
The exponential is related to the Hall pairing.
\begin{defn}
The Hall pairing $(-,-)$ is the unique scalar product on $\Sym$ such that $\{h_\lambda\}$, $\{m_\lambda\}$ form dual bases.
\end{defn}

Then $\pExp[XY]$ is the \emph{reproducing kernel} for the Hall pairing. This means the following
\begin{prop}[Reproducing kernel identity]
For any $F[X]\in\Sym[X]$ we have
$$
(F[X], \pExp[XY])_X = F[Y].
$$
\end{prop}
 
The subscript $X$ means we take the scalar product with respect to $X$. A consequence of this fact is that for any pair of dual bases $\{a_\lambda\}$, $\{b_\lambda\}$ we have
$$
\pExp[XY] = \sum_{\lambda} a_\lambda[X] b_\lambda[Y].
$$

\subsection{Coproduct}
Perhaps a more conceptual way to understand the Hall product is through the Hopf algebra structure.
\begin{defn}
The coproduct on $\Sym[X]$ is defined in the following way. We have
$$
\Delta:\Sym[X] \to \Sym[X]\otimes\Sym[X] \cong \Sym[X,Y]
\qquad 
(\Delta F)[X, Y] = F[X+Y].
$$
\end{defn}

The coproduct is dual to the product in the following way:
\begin{prop}\label{prop:coprod}
For each $F,G,G'\in\Sym[X]$ we have
$$
(F[X], G[X] G'[X])_X = (F[X+Y], G[X] G'[Y])_{X,Y}.
$$
\end{prop}
\begin{proof}
For any $F$ we can write
$$
F[X] = (\pExp[XZ], F[Z])_Z.
$$
We substitute this in the right hand side of the identity we need to prove and obtain
$$
((\pExp[(X+Y)Z], F[Z])_Z, G[X] G'[Y])_{X,Y} = (G[Z]G'[Z], F[Z])_Z
$$
by computing the scalar product with respect to $X$ and $Y$.
\end{proof}

This can be pictured as follows:

\vspace{.5cm}
\begin{center}
\begin{tikzpicture}
\node (1) {$G$};
\node (2) [below =1.3cm of 1] {$G'$};
\node (3) [below right=0.5cm and 0.5cm of 1] {$\cdot$}; 
\node (4) [right = 1cm of 3] {$F$};

\path[draw, thick, ->]
(1) edge node {} (3)
(2) edge node {} (3)
(3) edge node {} (4);

\node (5) [right=2cm of 4] {$=$};
\node (6) [right = 2cm of 5] {$F$};
\node (7) [right = 1cm of 6] {$\Delta$};
\node (8) [above right = 0.5cm and 0.5cm of 7] {$G$};
\node (9) [below =1.3cm of 8] {$G'$};

\path[draw, thick, ->]
(6) edge node {} (7)
(7) edge node {} (8)
(7) edge node {} (9);
\end{tikzpicture}
\end{center}

Finally an important property is 
\begin{prop}
For $F,G\in\Sym[X]$ we have
$$
(F[XY], G[X])_X = (F[X], G[XY])_X.
$$
\end{prop}
\begin{proof}
Similar to Proposition \ref{prop:coprod}
\end{proof}

\section{Macdonald polynomials}
Denote by $M_{\leq\lambda}$ the subspace of $\Sym$ spanned by $m_\mu$ for $\mu\leq \lambda$ with respect to the dominance order. Schur polynomials can be identified up to a scalar factor by the properties 
$$
s_\lambda[X] \in M_{\leq \lambda}, \qquad s_\lambda[-X] \in M_{\leq \lambda'},
$$
where $\lambda'$ denotes the conjugate partition. The Macdonald polynomials can be defined similarly. From now on $\Sym$ will be the symmetric functions over the field $\Q(q,t)$ with the standard $\lambda$-structure.

\begin{defn}[\cite{garsia1996remarkable}]
Macdonald polynomials\footnote{Usually these polynomials are denoted $\wt{H}_\lambda$. We omit the tilde from the notation.} $H_\lambda\in \Sym$ are unique polynomials such that
$$
H_\lambda[(t-1)X]\in M_{\leq \lambda},\qquad H_\lambda[(q-1)X]\in M_{\leq \lambda'},\qquad H_\lambda[1]=1.
$$
\end{defn}

\begin{example}
Directly from the definition we deduce the two special cases
$$
H_{(n)}[X] = \frac{e_n\left[\frac{X}{q-1}\right]}{e_n\left[\frac{1}{q-1}\right]}
=\frac{h_n\left[\frac{X}{1-q}\right]}{h_n\left[\frac{1}{1-q}\right]},
$$
$$
H_{1^n}[X] = \frac{e_n\left[\frac{X}{t-1}\right]}{e_n\left[\frac{1}{t-1}\right]}
=\frac{h_n\left[\frac{X}{1-t}\right]}{h_n\left[\frac{1}{1-t}\right]}.
$$
From the defining property of the Schur functions we deduce specializations
$$
H_\lambda[X]|_{t=q^{-1}} = \frac{s_\lambda\left[\frac{X}{1-q}\right]}{s_\lambda\left[\frac{1}{1-q}\right]} \qquad (\lambda\in\cP).
$$
\end{example}

It follows that if we define the Macdonald scalar product as follows
$$
(F[X], G[X])_*:=(F[X], G[QX])\qquad Q=(q-1)(1-t)
$$
the Macdonald polynomials become orthogonal with respect to the Macdonald scalar product. Let
$$
\alpha_\lambda:=(H_\lambda, H_\lambda)_* = \prod_{a,l}(q^a - t^{l+1})(q^{a+1} - t^l),
$$
where the product runs over the arm- and leg-lengths of the hooks of $\lambda$.

The reproducing kernel for the Macdonald scalar product is the element
$$
\pExp\left[\frac{XY}{Q}\right] = \sum_{\lambda\in\cP} \frac{H_\lambda[X] H_\lambda[Y]}{\alpha_\lambda}.
$$

\section{HLV conjecture}
Define for each $k=1,2,3,\ldots$
$$
\Omega[X_1,\ldots,X_k]:=\sum_{\lambda\in\cP} \frac{H_\lambda[X_1] H_\lambda[X_2]\cdots H_\lambda[X_k]}{\alpha_\lambda}.
$$
Also define $\bH$ by the formula
$$
\Omega[X_1,\ldots,X_k] = \pExp\left[\frac{\bH[X_1,X_2,\ldots,X_k]}{Q}\right].
$$
The genus $0$ version of the conjecture of Hausel, Letellier and Rodriguez-Villegas is
\begin{conjecture}[HLV, genus $0$]
All coefficients of the series $\bH$ are polynomials in $q$ and $t$, and for each tuple of partitions $\lambda^{(1)}$, $\lambda^{(2)}$,\ldots,$\lambda^{(k)}$ the scalar product
$$
\bH(h_{\lambda^{(1)}}, \ldots, h_{\lambda^{(k)}}):=(\bH[X_1,\ldots,X_k],\; h_{\lambda^{(1)}}[X_1] \cdots h_{\lambda^{(k)}}[X_k])_{X_1,\ldots,X_k}
$$
coincides with the mixed Hodge polynomial of the moduli space of local systems on $\PP^1$ with $k$ punctures whose local monodromies around the punctures are fixed diagonalizable conjugacy classes of types $\lambda^{(1)}$, $\lambda^{(2)}$,\ldots,$\lambda^{(k)}$ satsifying certain genericity condition.
\end{conjecture}

The precise formulation can be found in \cite{hausel2011arithmetic}.
\begin{rem}
We have the following identity.
$$
\bH[X_1,\ldots,X_k,1] = \bH[X_1,\ldots,X_k]
$$
\end{rem}

\begin{rem}
We have $\bH[X_1,X_2]=X_1 X_2$.
\end{rem}

\begin{defn}
The Macdonald coproduct $\Delta'$ is defined by 
$$
\Delta' H_\lambda = H_\lambda \otimes H_\lambda \qquad(\lambda\in\cP).
$$
\end{defn}

\begin{rem}
We can obtain $\Omega$ by iterated application of the Macdonald coproduct starting from the case $k=1$.
\end{rem}

\section{Plethystic identities}
The purpose of this section is to demonstrate some techniques in dealing with Macdonald polynomials. There are several works studying plethystic identities related to the Macdonald polynomials. The most important identity, from which every other identity can be easily deduced seems to be the following:
\begin{thm}[\cite{garsia1999explicit}, Theorem I.3]
Let $\nabla, T, T^*:\Sym\to\Sym$ be the operators
$$
\nabla H_\lambda = H_\lambda[-1]\cdot H_\lambda,\qquad (TF)[X] = F[X+1], \qquad (T^* F)[X] = \pExp\left[\frac{X}{Q}\right] F[X].
$$
Let $V$ be the composition
$$
V:=\nabla T^* T.
$$
For each partition $\lambda$ define the generating functions
$$
B_\lambda:=\sum_{r,c\in\lambda} q^c t^r,\qquad D_\lambda:=-1-Q B_\lambda.
$$
Then
$$
V(H_\lambda) = \pExp\left[\frac{D_\lambda X}{Q}\right]\qquad (\lambda\in\cP).
$$
\end{thm}

\begin{rem}
The operators $T$, $T^*$ are conjugate with respect to the Macdonald scalar product.
\end{rem}

\begin{rem}
Our operator $\nabla$ differs from the original one by the sign $(-1)^{|\lambda|}$. This choice is made in order to eliminate the operators ${^-}(\cdot)$ and $\omega$ from our notation, because they interact poorly with plethystic substitution.
\end{rem}

\begin{rem}
The theorem implies that $V$ transforms the Macdonald coproduct to the usual coproduct:
\vspace{.5cm}
\begin{center}
\begin{tikzpicture}
\node (1) {};
\node (2) [right = 1.5cm of 1]{$\Delta'$};
\node (3) [above right =1cm and 1cm of 2] {};
\node (4) [below right=1.1cm and 1cm of 2] {}; 

\path[draw, thick, ->,above]
(1) edge node {} (2)
(2) edge node {$V$} (3);
\path[draw, thick, ->,right]
(2) edge node {$V$} (4);

\node (5) [right=3cm of 2] {$=$};
\node (6) [right = 2cm of 5] {};
\node (7) [right = 1.5cm of 6] {$\Delta$};
\node (8) [above right = 1cm and 1cm of 7] {};
\node (9) [below right =1cm and 1cm of 7] {};

\path[draw, thick, ->, above]
(6) edge node {$V$} (7)
(7) edge node {} (8)
(7) edge node {} (9);
\end{tikzpicture}
\end{center}
\end{rem}

Another interesting statement can be obtained by conjugation:
\begin{prop}\label{prop:deltaop}
For each $F\in \Sym$ let $\Delta_F$ be the operator
$$
\Delta_F:\Sym\to\Sym\qquad \Delta_F H_\lambda = F[D_\lambda] H_\lambda.
$$
Let $M_F$ be the operator of multiplication by $F$. Then
$$
V^* M_F = \Delta_F V^*,
$$
where $V^*=T^* T \nabla$ is the operator conjugate to $V$.
\end{prop}
\begin{proof}
One way to see this is to mentally conjugate the pictures above. A more explicit way is to write both sides in the Macdonald basis. For each $\lambda, \mu\in\cP$
$$
(V^* (F H_\lambda),\; H_\mu)_* = F[D_\mu] H_\lambda[D_\mu],
$$
$$
(\Delta_F V^* H_\lambda,\; H_\mu)_* = F[D_\mu] (V^* H_\lambda,\; H_\mu)_*= F[D_\mu] H_\lambda[D_\mu].
$$
\end{proof}

Another statement is an interesting symmetry:
\begin{prop}[Macdonald-Koornwinder duality]
Consider
$$
\pExp[u B_\lambda]\, H_\lambda[1+u D_\mu] = \frac{H_\lambda[1+u D_\mu]}{\prod_{r,c\in\lambda} (1-u q^c t^r)},
$$
where $u$ is a new monomial variable. This expression is symmetric in $\lambda, \mu\in\cP$.
\end{prop}
\begin{proof}
It is equivalent to symmetry of the following expression:
$$
\pExp\left[\frac{u D_\mu}Q\right]\, H_\lambda[1+u D_\mu].
$$
We notice that this can be written as
$$
\left(T^* T H_\lambda,\; \pExp\left[\frac{uX D_\mu}Q\right]\right)_*
=\left((T^* T H_\lambda)[X],\; (\nabla T^* T H_\mu)[uX]\right)_*.
$$
Since $\nabla$ preserves the degree of symmetric functions and is self-adjoint, the last expression is symmetric in $\mu$ and $\lambda$.
\end{proof}

Setting $\mu=()$ we obtain
\begin{cor}\label{cor:specu}
$$
H_\lambda[1-u] = \pExp[-uB_\lambda] = \prod_{r,c\in\lambda} (1-u q^c t^r)\qquad(\lambda\in\cP).
$$
\end{cor}
In the limit $u\to\infty$ we obtain
\begin{cor}
$$
H_\lambda[-1] = (-1)^{|\lambda|} \prod_{r,c\in\lambda} q^c t^r = (-1)^{|\lambda|} q^{n(\lambda')} t^{n(\lambda)} \qquad(\lambda\in\cP).
$$
\end{cor}

\section{An interesting bijection}
In view of Corollary \ref{cor:specu} is seems natural to evaluate $\Omega$ on arguments of type $1-u$ where $u$ is monomial. Consider
$$
A:=\Omega[X, 1-y_1, 1-y_2,\ldots,1-y_k],
$$
where $y_1, y_2, \ldots$ are monomial variables. Set $Y=\sum_{i=1}^k y_i$. By Corollary \ref{cor:specu} we can write it as
$$
A=\sum_{\lambda\in\cP} \frac{H_\lambda[X] \pExp[-B_\lambda Y]}{\alpha_\lambda} = \pExp\left[\frac{Y}{Q}\right] \sum_{\lambda\in\cP} \frac{H_\lambda[X] \pExp\left[\frac{D_\lambda Y}{Q}\right]}{\alpha_\lambda}.
$$
Thus to obtain the sum on the right we simply need to apply $V$ to $\pExp\left[\frac{XY}Q\right]$ with respect to the variable $Y$. Since $\pExp\left[\frac{XY}Q\right]$ is the reproducing kernel, or simply by a direct computation we see that the result will be the same if we apply the conjugate operator $V^*$ with respect to the variable $X$:
$$
A = \pExp\left[\frac{Y}{Q}\right] V^* \pExp\left[\frac{XY}{Q}\right].
$$
Applying $\nabla$ to $\pExp\left[\frac{XY}Q\right]$ we obtain $\Omega[X, Y, -1]$, and then $T$ and $T^*$ can be computed directly:
$$
A = \pExp\left[\frac{X+Y}Q\right] \Omega[X+1, Y, -1].
$$

Thus we have
$$
\Omega[X, 1-y_1, 1-y_2,\ldots,1-y_k] = \pExp\left[\frac{X+Y}Q\right] \Omega[X+1, Y, -1].
$$
This is convenient because now $Q\cdot\pLog[-]$ can be taken from both sides.
$$
\bH[X, 1-y_1, 1-y_2,\ldots,1-y_k] = X+Y + \bH[X+1, Y, -1].
$$

For any symmetric function $F$ and monomial $u$ we have
$$
F[1-u] = (F[X],\, \pExp[X(1-u)]) = \left(F,\, \sum_{i,j=0}^\infty (-1)^j u^j h_i e_j\right),
$$
$$
F[-1] = \left(F,\, \sum_{i=0}^\infty (-1)^i e_i\right),
$$
$$
(F[X+1], h_\lambda[X]) = (F[X],\, \pExp[X] h_\lambda[X]) = \left(F, \sum_{i=0}^\infty h_{\lambda, i}\right).
$$
So we can interpret our result in the following way
\begin{thm}
For each pair of partitions $\lambda$, $\mu$ such that $1\leq\mu_1\leq |\lambda|\leq|\mu|$ we have
$$
\bH(h_\lambda, h_{|\lambda|-\mu_1}e_{\mu_1}, \ldots, h_{|\lambda|-\mu_k}e_{\mu_k}) = \bH(h_{\lambda, |\mu|-|\lambda|}, h_\mu, e_{|\mu|}).
$$
\end{thm}

An interesting special case is $\mu=1^k$. Because $e_1=h_1$ we obtain a somewhat simpler looking identity:
\begin{cor}\label{cor:final}
For integers $k\geq n \geq 1$ and $\lambda\vdash n$ we have
$$
\bH(h_\lambda,\, \underbrace{h_{n-1,1},\, \ldots,\, h_{n-1,1}}_\text{$k$ times})= \bH(h_{\lambda, k-n},\, h_{1^k},\, e_k).
$$
\end{cor}

This work was presented at the workshop ``Arithmetic aspects of moduli spaces'', held in 2016 at the Bernoulli center, Lausanne. During the talk T. Hausel pointed out that Corollary \ref{cor:final} solves a conjecture from their recent work \cite{hausel2016arithmetic}. In their work the statement of Corollary \ref{cor:final} is interpreted as the Laplace transform of wild character varieties.

\printbibliography

\end{document}